\documentclass[11pt]{amsart} 

\usepackage{a4wide,amssymb,color,textcomp}
\usepackage[all]{xy}

\makeatletter
\DeclareFontEncoding{LS1}{}{}
\DeclareFontSubstitution{LS1}{stix}{m}{n}
\DeclareMathAlphabet{\mathcal}{LS1}{stixscr}{m}{n}
\makeatother

\newcommand{\A}{\mathcal A}
\newcommand{\F}{\mathcal F}
\newcommand{\I}{\mathcal I}
\newcommand{\J}{\mathcal J}
\newcommand{\U}{\mathcal U}
\newcommand{\C}{\mathcal C}
\newcommand{\V}{\mathcal V}
\newcommand{\s}{\mathcal S}

\newcommand{\M}{\mathcal M}

\newcommand{\e}{\varepsilon}
\newcommand{\w}{\omega}
\newcommand{\diam}{\mathrm{diam}}
\newcommand{\IN}{\mathbb N}
\newcommand{\IB}{\mathbb B}
\newcommand{\IR}{\mathbb R}
\newcommand{\strategy}{\mbox{\rm\texteuro}}
\newcommand{\Ra}{\Rightarrow}

\newcommand{\se}{\subseteq}

\newcommand{\defeq}{\overset{\mbox{\tiny\sf def}}=}

\newtheorem{theorem}{Theorem}[section]
\newtheorem{lemma}[theorem]{Lemma}
\newtheorem{claim}[theorem]{Claim}
\newtheorem{corollary}[theorem]{Corollary}
\newtheorem{proposition}[theorem]{Proposition}
\newtheorem{problem}[theorem]{Problem}

\theoremstyle{definition}
\newtheorem{definition}[theorem]{Definition}
\newtheorem{remark}[theorem]{Remark}
\newtheorem{example}[theorem]{Example}

\title{The Set-Cover game and non-measurable unions}
\author{Taras Banakh, Robert Ra\l owski, Szymon \.Zeberski}

\address{T.Banakh: Ivan Franko National University of Lviv (Ukraine) and\newline Jan Kochanowski University in Kielce (Poland)}
\email{t.o.banakh@gmail.com}

\address{R.Ra\l owski, S.\.Zeberski: Wroc\l aw University of Science and Technology,\newline Wybrze\.ze S.~Wyspia\'nskiego 27, 50-370 Wrocław}
\email{robert.ralowski@pwr.edu.pl, szymon.zeberski@pwr.edu.pl}

\thanks{The work of R. Rałowski and Sz. Żeberski has been partially financed by grant { 8211204601, MPK: 9120730000} from the Faculty of Pure and Applied Mathematics, Wrocław University of Science and Technology.}


\begin{document}

\begin{abstract} Using a game-theoretic approach we present a generalization of the classical result of Brzuchowski, Cicho\'n, Grzegorek and Ryll-Nardzewski on non-measurable unions.  We also present applications of obtained results to Marczewski--Burstin representable ideals.
\end{abstract}
\maketitle

\section{Introduction and selected results}

In this paper we generalize the following classical result on non-measurable unions, proved by Brzuchowski, Cicho\'n, Grzegorek and Ryll-Nardzewski in \cite{BCGRN}.

\begin{theorem}\label{t:BCGRN} For any $\sigma$-ideal $\I$ with a Borel base on a Polish space $X$,  any point-finite family $\J\subseteq\I$ with $\bigcup\J\notin\I$ contains a subfamily $\J'\subseteq \J$ whose union $\bigcup\J'$ does not belong to the $\sigma$-algebra generated by Borel sets in $X$ and sets in the ideal $\I$.
\end{theorem}

A family of sets $\J$ is {\em point-finite} if $\{J\in\J: x\in J\}$ is finite for every $x\in \bigcup\J$. 

In this paper we present a generalization of Theorem~\ref{t:BCGRN} with a game-theoretic proof, which will allow us establish various additional non-measurability properties of the union $\bigcup\J'$.
Our proof exploits the Set-Cover game defined as follows.

Let $\I$ be a $\sigma$-ideal on a set $X$ and $\A$ be a multiplicative family of subsets of $X$ such that $X\in\A\setminus\I$. The multiplicativity of $\A$ means that $A\cap B\in\A$ for any sets $A,B\in\A$. 

The Set-Cover game $\Game_{\A{\setminus}\I\!,\A}$ is played by two players $\mathsf S$ and $\mathsf C$ (abbreviated from {\sf Set} and {\sf Cover}). Player $\mathsf S$ starts the game selecting a set $S_1\in\A\setminus\I$. Player $\mathsf C$ answers with a countable cover $\C_1\subseteq\A$ of the set $S_1$. At the $n$th inning Player $\mathsf S$ chooses a set $S_n\in\A\setminus\I$ such that $S_n$ is contained in some set of the cover $\mathcal C_{n-1}$ of $S_{n-1}$ and Player $\mathsf C$ answers with a countable cover $\mathcal C_n\subseteq\A$ of the set $S_n$. 
At the end of the game Player $\mathsf C$ is declared the winner of the Set-Cover game $\Game_{\A{\setminus}\I\!,\A}$ if the intersection $\bigcap_{n\in\w}S_n$ is not empty. In the opposite case, Player $\mathsf S$ wins the game ${\mathbf\Game}_{\A{\setminus}\I\!,\A}$.

The family $\A$ is called {\em $\I$-winning} if Player $\mathsf C$ has a winning strategy in the Set-Cover game  ${\mathbf\Game}_{\A{\setminus}\I\!,\A}$. 
 A set $A\subseteq X$ is called {\em $\I$-positive} if $A\notin\I$.
\smallskip

One of main results of this paper is the following theorem.

\begin{theorem}\label{t:main1} Let $\I$ be a $\sigma$-ideal on a set $X$ of cardinality $|X|\le\mathfrak c$ and $\A$ be an $\I$-winning multiplicative family of subsets of $X\in\A\setminus\I$. Every point-finite family $\J\subseteq\I$ with $\bigcup\J\notin\I$ contains a subfamily $\J'\subseteq\J$ such that for some set $A\in\A$ the intersection $A\cap\bigcup\J'$ is $\I$-positive but no subset of $A\cap\bigcup\J'$ belongs to the family $\A\setminus\I$.
\end{theorem}

Theorem~\ref{t:main1} motivates the problem of recognizing $\I$-winning families $\A$. Important examples of such families are Polishable families, defined as follows.

\begin{definition} A family $\A$ of subsets of a topological space is {\em Polishable} if for every $A\in\A$ there exist a zero-dimensional Polish space $P$ and a continuous surjective map 
$f:P\to A$ such that $\{f[B]:B\in\mathcal B\}\subseteq\A$ for some base $\mathcal B$ of the topology of $P$.
\end{definition}

The following theorem will be proved in Section~\ref{s:Polishable}.

\begin{theorem}\label{t:Polish} Every Polishable multiplicative  family $\A$ of subsets of a Hausdorff space $X$ is $\I$-winning for any ideal $\I$ on $X$.
\end{theorem} 

In Section~\ref{s:Polishable} we shall prove that for any Polish space $X$ the descriptive classes $\mathbf\Sigma^1_1(X)$, $\mathbf\Delta^1_1(X)$, and $\mathbf\Sigma^0_\alpha(X)$,  $\mathbf\Pi^0_\alpha(X)$, $\mathbf\Delta^0_\alpha(X)$ for $2\le\alpha<\w_1$ are Polishable.

Applying Theorems~\ref{t:main1} and \ref{t:Polish} to the Polishable family $\mathbf\Delta^1_1(X)$ of Borel sets in a Polish space $X$, we obtain the following improvement of Theorem~\ref{t:BCGRN}.

\begin{theorem}\label{t:main-B} Let $\I$ be a $\sigma$-ideal on a Polish space $X$. Any point-finite family $\J\subseteq\I$ with $\bigcup\J\notin\I$  contains a subfamily $\J'\subseteq\J$ such that for some Borel set $B$ in $X$ the set $B\cap\bigcup\J'$ is $\I$-positive but contains no $\I$-positive Borel subsets of $X$. 
\end{theorem}

Applying Theorems~\ref{t:main1} and \ref{t:Polish} to the Polishable class $\mathbf\Sigma^1_1(X)$ of analytic sets in a Polish space $X$, we obtain the following   ``analytic'' nonmeasurability criterion, which has been applied in a recent solution \cite{Ban} of an old problem of the automatic continuity of universally (resp. Haar) measurable homomorphisms on (locally compact) \v Cech-complete topological groups.

\begin{theorem}\label{t:main-A} Let $\I$ be a $\sigma$-ideal on a Polish space $X$. Any point-finite family $\J\subseteq\I$ with $\bigcup\J\notin\I$  contains a subfamily $\J'\subseteq\J$ such that for some analytic set $A$ in $X$ the set $A\cap\bigcup\J'$ is $\I$-positive but contains no $\I$-positive analytic subsets. 
\end{theorem}

Theorem~\ref{t:main1} will be derived in Section~\ref{s:=>} from a much more precise and elaborated Theorem~\ref{t:main}, which implies also some new results for the classical ideals related to measure and category, see Corollaries~\ref{c:measure} and \ref{c:category}.

\begin{theorem}\label{t:measure} Let $\lambda$ be a non-atomic probability $\sigma$-additive Borel measure on a Polish space $X$ and $\mathcal I\defeq\{A\subseteq X:\lambda^*(A)=0\}$.  For any point-finite subfamily $\J\subseteq\mathcal I$  and any $\e>0$ there exist a Borel set $A\subseteq X$ of measure $\lambda(A)>1-\e$ and a finite partition $\J=\J_1\cup\dots\cup\J_n$ of the set $\J$ such that $\lambda_*(A\cap\bigcup\J_i)=0$ for all $i\in\{1,\dots,n\}$. 
\end{theorem}

\begin{theorem}\label{t:category} Let $X$ be a crowded compact metrizable space and $\I$ be the $\sigma$-ideal of meager sets in $X$. For any point-finite subfamily $\J\subseteq\I$ there exists a closed nowhere dense set $D\subseteq X$ such that for any neighborhood $U$ of $D$ in $X$ there exists a finite partition $\J=\J_1\cup\dots\cup\J_n$ of the set $\J$ such that for every $i\in\{1,\dots,n\}$  the set $\bigcup\J_i\setminus U$ contains no non-meager Borel subsets of $X$.
\end{theorem}

Let us recall that a topological space is {\em crowded} if it has no isolated points. A subset of a topological space is {\em regular closed} if it is equal to the closure of its interior.

In Section~\ref{s:mb} we shall apply Theorem~\ref{t:main-B} to Marczewski--Burstin representable ideals and algebras, which include the  Marczewski ideal $s_0$ and the algebra of $s$-measurable sets, the $t_0$ ideal and the algebra of $t$-measurable sets (for various classes of perfect trees), the ideal of Lebesgue null sets and the algebra of Lebesgue measurable sets, the ideal of meager sets and the algebra of comletely Ramsey sets in $[\w]^\w$, etc.


\section{Preliminaries}

Let $\w$ denote the set of all finite ordinals and $\IN\defeq\w\setminus\{0\}$ be the set of positive integers. Each number $n\in\w$ is the set $\{0,\dots,n-1\}$ of smaller ordinals. A set is {\em countable} if it admits an injective function into $\w$.

For a set $A$, let  $[A]^{\le \w}\defeq\{S\subseteq A:|S|\le\w\}$ and $A^{<\w}\defeq\bigcup_{n\in\w}A^n$. For any $s\in A^{<\w}$ let
$$A^{<\w}_s\defeq\{t\in A^{<\w}:s\subseteq t\}\mbox{ \ and \ }A^{\w}_s\defeq\{t\in A^{\w}:s\subset t\}.$$

For any sequence $s=(s_i)_{i\in n}\in A^n\subseteq A^{<\w}$ we denote by $|s|$ its length $n$.

The set $A^{<\w}$ is a tree with respect to the partial order $\subseteq$. Two elements $s,t\in A^{<\w}$ are called {\em comparable} if $s\subseteq t$ or $t\subseteq s$. In the other case, the elements $s,t$ are called {\em incomparable}. 

A subset $L\subseteq A^{<\w}$ is called a {\em chain} if $L$ is linearly ordered by the relation $\subseteq$ (which means that for any $s,t\in L$ either $s\subseteq t$ or $t\subseteq s$).

 A nonempty subset $T\subseteq A^{<\w}$ is called a {\em subtree} if for any $t\in T$ we have $\{t{\restriction}n:n\in |t|\}\subseteq T$. A subtree $T\subseteq A^{<w}$ is called {\em perfect} if for every $t\in T$ the set $T\cap A^{<\w}_t$ contains at least two incomparable elements.

A {\em Cantor scheme} is an indexed family of sets $(X_s)_{s\in 2^{<\w}}$ such that for every $s\in 2^{<\w}$ we have $X_{s\hat{\;}0}\cup X_{s\hat{\;}1}=X_s$ and $X_{s\hat{\;}0}\cap X_{s\hat{\;}1}=\emptyset$.
\smallskip
 
For a family of sets $\A$, let $\sigma\A\defeq\{\bigcup\mathcal B:\mathcal B\in[\A]^{\le\w}\}$ be the family of countable unions of sets in $\A$. Also for any set $S$, let $$\A(S)\defeq\{A\in\A:A\subseteq S\}.$$

Given a set $S$, we write $S\prec\A$ if $S\subseteq A$ for some set $A\in\A$.
\smallskip

An {\em ideal} on a nonempty set $X$ is a nonempty family $\I$ of subsets of $X$ such that 
$X\notin\I$ and for any sets $I,J\in\I$, any subset of the union $I\cup J$ belongs to $\I$. The family $\{\emptyset\}$ is the smallest ideal on $X$. An ideal $\I$ on $X$ is called a {\em $\sigma$-ideal} if $\sigma\I=\I$.

A family $\A$ is called a {\em $\sigma$-algebra} on a set $X$ if $X=\bigcup\A$ and $X\setminus\bigcup\C\in\A$ for any countable subfamily $\C\subseteq\A$.  

A family of sets $\A$ is called 
\begin{itemize}
\item {\em disjoint} if $A\cap B=\emptyset$ for any distinct sets $A,B\in\A$;
\item {\em multiplicative} if $A\cap B\in\A$ for any sets $A,B\in\A$;
\item {\em point-finite} if for every element $x\in\bigcup\A$ the family $\{A\in \A:x\in A\}$ is finite.
\end{itemize}
For two sets $A,B$ let $A\Delta B\defeq (A\setminus B)\cup(B\setminus A)=(A\cup B)\setminus(A\cap B)$ be their {\em symmetric difference}.

\section{$\I$-Lindel\"of and $\I$-ccc families}

\begin{definition} 
Let $\I$ be an ideal on a set $X$. A family $\A$ of subsets of  $X$ is called
\begin{itemize}
\item {\em $\I$-Lindel\"of} if for any subset $S\subseteq X$, there exists a countable subfamily $\C\subseteq\A(S)$ such that $\A(S\setminus\bigcup\C)\subseteq\I$;
\item {\em $\I$-ccc} if each disjoint subfamily of $\A\setminus\I$ is countable.
\end{itemize}
\end{definition}

\begin{lemma}\label{l:ccc} Let $\I$ be an ideal on a set $X$ and $\A$ be a family of subsets of $X$. If $\A$ is $\I$-ccc, then $\A$ is $\I$-Lindel\"of.
\end{lemma}

\begin{proof}  Assume that $\A$ is $\I$-ccc. Given any subset $S\subseteq X$, use the Kuratowski--Zorn Lemma to find a maximal disjoint family $\mathcal F$ in the family $\A(S)\setminus\I$. Since $\A$ is $\I$-ccc, the family $\mathcal F$ is at most countable. The maximality of the family $\mathcal F$ ensures that $\A(S\setminus \bigcup\F)\subseteq\I$. Therefore, the countable family $\F$ witnesses that the family $\A$ is $\I$-Lindel\"of.
\end{proof}

\begin{problem}\label{prob1} Let $\I$ be a $\sigma$-ideal and $\A$ be an $\I$-Lindel\"of $\sigma$-algebra of subsets of $X$. Is $\A$ $\I$-ccc?
\end{problem}

Below we shall give two partial affirmative answers to Problem~\ref{prob1}.

\begin{proposition} Let $\I$ be an ideal on a set $X$ and $\A$ be a family of subsets of $X$. If $|\sigma\A|<2^{\w_1}$, then the family $\A$ is $\I$-Lindel\"of if and only if it is $\I$-ccc.
\end{proposition}

\begin{proof} The ``if'' part follows from Lemma~\ref{l:ccc}. To prove the ``only if'' part, assume that $\A$ is $\I$-Lindel\"of but not $\I$-ccc. Then there exists an uncountable family of pairwise disjoint sets $\{B_\alpha\}_{\alpha\in\w_1}\subseteq\A\setminus\I$. By the $\I$-Lindel\"of property, for any subset $\Omega\subseteq\w_1$, there exists a set $\Sigma_\Omega\in\sigma\A(\bigcup_{\alpha\in\Omega}B_\alpha)$ such that $\A(\bigcup_{\alpha\in\Omega}B_\alpha\setminus\Sigma_\Omega)\subseteq\I$. Since $|\sigma\A|<2^{\w_1}$, there are two distinct subsets $\Omega,\Lambda\subseteq\w_1$ such that $\Sigma_\Omega=\Sigma_\Lambda$. Choose any ordinal $\beta\in \Omega\triangle\Lambda$. If $\beta\in\Omega\setminus\Lambda$, then 
$\emptyset=B_\alpha\cap \Sigma_\Lambda=B_\alpha\cap \Sigma_\Omega$ and hence $B_\alpha\in\A(\bigcup_{\alpha\in\Omega}B_\alpha\setminus\Sigma_\Omega)\setminus\I$, which contradicts the choice of the set $\Sigma_\Omega$. If $\beta\in \Lambda\setminus\Omega$, then   $B_\alpha\in\A(\bigcup_{\alpha\in\Lambda}B_\alpha\setminus\Sigma_\Lambda)\setminus\I$, which contradicts the choice of the set $\Sigma_\Lambda$.
\end{proof}

Another partial answer to Problem~\ref{prob1} can be obtained by application of Theorem~\ref{t:main1}.

\begin{proposition} Let $\I$ be a $\sigma$-ideal on a set $X$ of cardinality $|X|\le\mathfrak c$ and $\A$ be a $\I$-winning $\sigma$-algebra on $X$. The family $\A$ is $\I$-ccc if and only if $\A$ is $\I$-Lindel\"of.
\end{proposition}

\begin{proof} If $\A$ is $\I$-ccc, then $\A$ is $\I$-Lindel\"of by Lemma~\ref{l:ccc}. Now assume that $\A$ is $\I$-Lindel\"of. Assuming that $\A$ is not $\I$-ccc, we can find an uncountable family of pairwise disjoint sets $\{B_\alpha\}_{\alpha\in\w_1}\subseteq\A\setminus\I$. Consider the $\sigma$-ideal
$$\J\defeq\{J\subseteq X:\exists \alpha\in\w_1\;\forall \beta\in[\alpha,\w_1)\;\;J\cap B_\alpha\in \I\}.$$
It is clear that $\I\subseteq\J$, $\{B_\alpha\}_{\alpha\in\w_1}\subseteq\J$ and $\bigcup_{\alpha\in\w_1}B_\alpha\notin\J$. Since $\I\subseteq\J$, the $\I$-winning property of the family $\A$ implies the $\J$-winning property of $\A$. By Theorem~\ref{t:main1}, there exist a set $A\in\A$ and a subset $\Omega\subseteq\w_1$ such that $A\cap\bigcup_{\alpha\in\Omega}B_\alpha\notin\J$ and $\A(A\cap \bigcup_{\alpha\in\Omega}B_\alpha)\subseteq\J$. 
By the $\I$-Lindel\"of property of the $\sigma$-algebra $\A$, there exists a set $\Sigma\in\A(\bigcup_{\alpha\in \Omega}B_\alpha)$ such that  $\A(\bigcup_{\alpha\in\Omega}B_\alpha\setminus \Sigma)\subseteq\I$. 
We claim that the set $J\defeq\bigcup_{\alpha\in\Omega}B_\alpha\setminus\Sigma$ belongs to the ideal $\J$. This will follow as soon as we check that $J\cap B_\alpha\in\I$ for every $\alpha\in\w_1$. If $\alpha\in\w_1\setminus\Omega$, then $J\cap B_\alpha$ is empty and hence belongs to $\I$. So, we assume that $\alpha\in\Omega$. Taking into account that the sets $B_\alpha$ and $\Sigma$ belong to the $\sigma$-algebra $\A$,  we conclude that $B_\alpha\cap J=B_\alpha\setminus \Sigma\in\A(\bigcup_{\gamma\in\Omega}B_\gamma\setminus\Sigma)\subseteq\I$.
This completes the proof of the inclusion $J\in\J$.

  It follows from $J\in\J$ and $A\cap\bigcup_{\alpha\in\Omega}B_\alpha\notin\J$ that $A\cap\Sigma=A\cap\bigcup_{\alpha\in\Omega}B_\alpha\setminus J\notin\J$.  On the other hand, $A\cap\Sigma\in \A(A\cap\bigcup_{\alpha\in\Omega}B_\alpha)\subseteq\J$ implies $A\cap\Sigma\in\J$. This contradiction shows that the family $\A$ is $\I$-ccc.
\end{proof} 

\section{Polishable and $\I$-winning families}\label{s:Polishable}

In this section we prove that Polishable families are $\I$-winning. 

Let us recall that a class $\A$ of subsets of a topological space $X$ is {\em Polishable} if for any set $A\in\A$ there exist a Polish space $P$, a countable base $\mathcal B$ of the topology of $P$ and a continuous surjective map $f:P\to X$ such that $f[B]\in\A$ for every basic set $B\in\mathcal B$.

A  multiplicative family $\A$ of subsets of a set $X$ is {\em $\I$-winning} for an ideal $\I$ of subsets of $X$ with $X\in\A\setminus\I$ if the player $\mathsf C$ has a winning strategy in the Set-Cover game $\Game_{\A\setminus\I\!,\A}$.

A {\em strategy} of Player $\mathsf C$ in the game $\Game_{\A\setminus\I\!,\A}$  is a function $\strategy\colon(\A\setminus\I)^{<\w}\to[\A]^{\le\w}$ such that  $\strategy(\emptyset)=\{X\}$ and $S_n=\bigcup\strategy(S_0,\dots,S_n)$ for every nonempty finite sequence $(S_0,\dots,S_n)\in\mathcal (\A\setminus\I)^{<\w}$ such that $S_i\prec\strategy(S_0,\dots,S_{i-1})$ for all $i\le n$. 

A strategy $\strategy\colon(\mathcal A\setminus\I)^{<\w}\to[\A]^{\le\w}$ is {\em winning} if $\bigcap_{n\in\w}S_n\ne\emptyset$ for any sequence $(S_n)_{n\in\w}\in(\A\setminus\I)^\w$ such that $S_n\prec\strategy(S_0,\dots,S_{n-1})$ for all $n\in\w$. 

The following theorem is the main result of this section.

\begin{theorem}\label{t:Polishable} Let $\A$ be a multiplicative class of subsets of a Hausdorff topological space $X\in\A$. If $\A$ is  Polishable, then $\A$ is $\I$-winning for every ideal $\I$ on $X$.
\end{theorem}

\begin{proof}  Since the class $\A$ is Polishable, for every set $A\in\A$ there exist a Polish space $P_A$, a countable base $\mathcal B_A$ of the topology of $P_A$ and a continuous surjective map $f_A:P_A\to X$ such that  $\{f[B]:B\in\mathcal B\}\subseteq\A$. We can (and will) additionally assume that the base $\mathcal B_A$ consists of nonempty sets and $P_A\in\mathcal B_A$. Let $\rho_A\le 1$ be a complete metric generating the topology of the Polish space $P_A$. For every $n\in\w$, let $\mathcal B_{n,A}\defeq\{B\in\mathcal B_A:\diam_{\rho_A}(B)\le 2^{-n}\}$. Let $\mathcal B\defeq\bigcup_{A\in\A}\mathcal B_A$. The completeness of the metric $\rho_A$ ensures that for any sequence $(B_n)_{n\in\w}\in\prod_{n\in\w}\mathcal B_{n,A}$ with $\overline B_{n+1}\subseteq B_n$ for all $n\in\w$, the intersection $\bigcap_{n\in\w}B_n=\bigcap_{n\in\w}\overline B_n$ is a singleton.

Now we describe a winning strategy for the player $\mathsf C$ in the Set-Cover game $\Game_{\A\setminus\I\!,\A}$. Player $\mathsf S$ starts the game selecting a set $S_0\in\A\setminus\I$. Player $\mathsf C$ answers with the countable cover $\strategy(S_0)=\{f_{S_0}[B_0]:B_0\in\mathcal B_{1,S_0}\}$ of $S_0$. Then player $\mathsf S$ selects a set $S_1\in\A\setminus\I$ with $S_1\prec\strategy(S_0)$. The latter condition allows the player $\mathsf C$ to find a set $C_1\in\strategy(S_0)$ with $S_1\subseteq C_1\subseteq S_0$, then choose a set $\Phi_0(S_0,S_1)\in \mathcal B_{1,S_0}$ with $C_1=f_{S_0}[\Phi_0(S_0,S_1)]$ and finally suggest the countable cover $$\strategy(S_0,S_1)\defeq\big\{f_{S_0}[B_0]\cap f_{S_1}[B_1]:B_0\in\mathcal B_{2,S_0},\;\overline B_0\subseteq \Phi_0(S_0,S_1),\;\;B_1\in\mathcal B_{2,S_1}\big\}$$ of the set $S_1$ to the player $\mathsf S$ who selects a set $S_2\in\A\setminus\I$ with $S_2\prec\strategy(S_0,S_1)$. 

The latter condition allows the player $\mathsf C$ to find a set $C_2\in\strategy(S_0,S_1)$ containing $S_2$ and then choose sets $\Phi_0(S_0,S_1,S_2)\in \mathcal B_{2,S_0}$ and $\Phi_1(S_0,S_1,S_2)\in\mathcal B_{2,S_1}$ such that $\overline{\Phi_0(S_0,S_1,S_2)}\subseteq \Phi_0(S_0,S_1)$ and $C_2=f_{S_0}[\Phi_0(S_0,S_1,S_2)]\cap f_{S_1}[\Phi_1(S_0,S_1,S_2)]$. After doing this preparatory work, the player $\mathsf C$ suggests the countable cover $$\strategy(S_0,S_1,S_2)\defeq\Big\{\bigcap_{i=0}^2 f_{S_i}[B_i]:\forall i\;\,\big(B_i\in\mathcal B_{3,S_i}\;\wedge\;\big(i<2\;\Ra\;\overline B_i\subseteq \Phi_i(S_0,S_1,S_2)\big)\big)\Big\}$$ of the set $S_2$ to the player $\mathsf S$ who selects a set $S_3\in\A\setminus\I$ with $S_3\prec\strategy(S_0,S_1,S_2)$ and so on.
\smallskip

More formally, we can construct inductively two functions 
$\strategy:(\A\setminus\I)^{<\w}\to[\A]^{\le\w}$ and $\Phi:(\A\setminus\I)^{<\w}\to \mathcal B^{<\w}$ such that $\strategy(\emptyset)=\{X\}$, $\Phi(\emptyset)=\emptyset\in\mathcal B^\emptyset$ and for every sequence $(S_0,\dots,S_n)\in(\A\setminus\I)^{<\w}$ of positive length the following conditions are satisfied:
{\parindent=12pt\parskip3pt

(1) $\Phi(S_0,\dots,S_n)=\big(\Phi_i(S_0,\dots,S_n)\big)_{i\in n}\in \prod_{i\in n}\mathcal B_{n,S_i}$;

(2) if $S_n\prec \strategy(S_0,\dots,S_{n-1})$, then $S_n\subseteq\bigcap_{i\in n} f_{S_i}[\Phi_i(S_0,\dots,S_n)]$;

(3) $\overline{\Phi_i(S_0,\dots,S_n)}\subseteq \Phi_i(S_0,\dots,S_{n-1})$ for every $i\in n$;

(4) $\strategy(S_0,...\,,S_n)\defeq\big\{\bigcap\limits_{i=0}^n f_{S_i}[B_i]:\forall i\;\big(B_i\in\mathcal B_{n+1,S_i}\;\wedge\;\big(i<n\,\Ra\,\overline B_i\subseteq \Phi_i(S_0,...\,,S_{n-1})\big)\big)\big\}$.
}

The multiplicativity of the class $\A$ and the choice of the functions $f_A$ and bases $\mathcal B_A$ guarantee that the function $\strategy$ is a well-defined strategy of the player $\mathsf C$ in the Set-Cover game $\Game_{\A\setminus\I\!,\A}$.

It remains to prove that this strategy is winning. Take any $\strategy$-admissible sequence $(S_n)_{n\in\w}\in (\A\setminus\I)^\w$. The $\strategy$-admissibility of this sequence means that $S_n\prec \strategy(S_0,\dots,S_{n-1})$ for every $n\in\IN$.  The inductive conditions (1) and (4)  imply that for every $k\in\IN$ the intersection $\bigcap_{n=k}^\infty\Phi_k(S_0,\dots,S_n)$ is not empty and contains a unique point $z_k\in P_{S_k}$. Moreover, every neighborhood of $z_k$ contains all but finitely many sets $\Phi_k(S_0,\dots,S_n)$. The continuity of the map $f_{S_k}:P_{S_k}\to S_k$ and the inductive condition (2) imply that every neighborhood of the point $f_{S_k}(z_k)$ in $S_k$ contains all but finitely many sets $S_n$. This fact and the Hausdorff property of $X$ imply that the set $\{f_{S_k}(z_k):k\in\IN\}$ is a singleton and hence the intersection $\bigcap_{k\in\IN}S_k\supseteq \{f_{S_k}(z_k):k\in\IN\}$ is not empty.
\end{proof}

\begin{definition} 
A topological space $X$ is called
\begin{itemize}
\item {\em functionally Hausdorff} if for any distinct points $x,y\in X$ there exists a continuous function $f:X\to\mathbb R$ such that $f(x)\ne f(y)$;
\item {\em analytic} if $X$ is functionally Hausdorff and there exists a continuous surjective function $f:P\to X$ defined on some Polish space $P$;
\item {\em Borel} if $X$ is functionally Hausdorff and there exists a continuous  bijective function $f:P\to X$ defined on some Polish space $P$.
\end{itemize}
\end{definition}

\begin{proposition}\label{p:Borel2} For every Borel space $X$, the family $\mathbf\Delta^1_1(X)$ of Borel subspaces of $X$ coincides with the $\sigma$-algebra of Borel subsets of $X$.
\end{proposition}

\begin{proof} Since the space $X$ is Borel, there exists a Polish space $P$ and a continuous bijective map $f:P\to X$. Then the space $X$ has countable network and hence the subspace $D=\{(x,y)\in X\times X:x\ne y\}$ also has a countable network and hence is Lindel\"of.

Since the Borel space $X$ is functionally Hausdorff, for any distinct points $x,y\in X$ there exists a continuous function $g_{x,y}:X\to \IR$ such that $g_{x,y}(x)\ne g_{x,y}(y)$. By the continuity of $g_{x,y}$, the pair $(x,y)\in D$ has an open neighborhood $O_{x,y}\subseteq D$ such that $g_{x,y}(x')\ne g_{x,y}(y')$ for any $(x',y')\in O_{x,y}$.
Since the space $D$ is  Lindel\"of, the open cover $\{O_{x,y}:(x,y)\in D\}$ of $D$ has a countable subcover $\{O_{x,y}:(x,y)\in C\}$. Here $C$ is a suitable countable subset of $D$. Then the continuous function $$g:X\to \mathbb R^C,\quad g:z\mapsto (g_{x,y}(z))_{(x,y)\in C},$$ is injective.
\smallskip

Now we ready to prove that the family $\mathbf\Delta^1_1(X)$  of Borel subspaces of $X$ coincides with the $\sigma$-algebra $\A$
of Borel subsets of $X$. First we prove that $\A\subseteq\mathbf\Delta^1_1(X)$. 

Given any Borel set $A$ in $X$, we obtain that $f^{-1}[A]$ is a Borel subset of the Polish space $P$. By \cite[Theorem 13.1]{Ke}, there exists a Polish space $Q$ and a continuous bijective function $g:Q\to f^{-1}[A]$. Then the continuous bijective function $f\circ g:Q\to A$ witnesses that $A\in\mathbf\Delta^1_1(X)$.

Now take any set $A\in\mathbf\Delta^1_1(X)$. Then $A=h[S]$ for some continuous bijective function $h:S\to A$ defined on a Polish space $S$. Since the function $g\circ h:S\to\IR^C$ is injective, we can apply Lusin--Souslin Theorem \cite[Theorem 15.1]{Ke} to conclude that the image $g\circ h[S]$ is a Borel subset of the Polish space $\IR^C$. Then $A=g^{-1}[g\circ h[S]]$ is a Borel subset of $X$, being the preimage of a Borel set under a continuous map.
\end{proof} 

\begin{proposition}\label{p:Borel} For any Borel space $X$, the family $\mathbf\Delta^1_1(X)$ of Borel subspaces of $X$ is Polishable, multiplicative, and $\I$-winning for any $\sigma$-ideal $\I$ on $X$.
\end{proposition} 

\begin{proof} The Polishability of the family $\mathbf\Delta^1_1(X)$ follows from the definition of a Borel space. 

To see that the family $\mathbf\Delta^1(X)$ is multiplicative,
take any Borel subspaces $B_1,B_2\subseteq X$. For every $i\in\{1,2\}$, find a Polish space $P_i$ and a continuous bijective map $f_i:P_i\to B_i$. The Hausdorff property of the space $X$ ensures that the set $P=\{(x,y)\in P_1\times P_2:f_1(x)=f_2(y)\}$ is closed in the Polish space $P_1\times P_2$. Observe that $B_1\cap B_2$ is the image of the Polish space $P$ under the continuous bijective map $f:P\to B_1\cap B_2$, $f:(x,y)\mapsto f_1(x)=f_2(y)$.

By Theorem~\ref{t:Polishable}, the family $\mathbf\Delta^1_1(X)$ is $\I$-winning for any $\sigma$-ideal $\I$ on $X$.
\end{proof}

By analogy with Proposition~\ref{p:Borel} we can prove the following proposition.

\begin{proposition}\label{p:A-Polishable} For any analytic space $X$, the family $\mathbf\Sigma^1_1(X)$ of analytic subspaces of $X$ is Polishable, multiplicative, and $\I$-winning for any $\sigma$-ideal $\I$ on $X$.
\end{proposition}

Let us recall \cite[11.A]{Ke} that $\mathbf\Sigma^0_1(X)$ and $\mathbf\Pi^0_1(X)$ denote the families of open and closed subsets of a topological space $X$, respectively. For every countable ordinal $\alpha\ge 2$ the Borel classes $\mathbf\Sigma^0_\alpha(X)$ and $\mathbf\Pi^0_\alpha(X)$ are defined by the recursive formulas:
$$ 
\begin{aligned}
\mathbf\Sigma^0_\alpha(X)&\defeq\textstyle\big\{\bigcup\F:\F\subseteq \bigcup_{\beta<\alpha}\mathbf\Pi^0_\beta(X),\;\;|\F|\le\w\big\};\\
\mathbf\Pi^0_\alpha(X)&\defeq\textstyle\big\{\bigcap\F:\F\subseteq \bigcup_{\beta<\alpha}\mathbf\Sigma^0_\beta(X),\;\;|\F|\le\w\big\}.
\end{aligned}
$$
Let also $\mathbf\Delta^0_\alpha(X)\defeq\mathbf\Sigma^0_\alpha(X)\cap\mathbf\Pi^0_\alpha(X)$.

\begin{proposition} \label{p:polishable}For any Polish space $X$ the families
$\mathbf\Sigma^0_\alpha(X),\mathbf\Pi^0_\alpha(X)$ for $1\le\alpha<\w_1$ and $\mathbf\Delta^0_\alpha(X)$ for $2\le\alpha<\w_1$ are  multiplicative and  Polishable. The family $\mathbf \Delta^0_1(X)$ is Polishable if and only if the Polish space $X$ is zero-dimensional.
\end{proposition}

\begin{proof}  Let $X$ be a Polish space. By \cite[Proposition 22.1]{Ke}, for every countable ordinal $\alpha\ge 1$ the classes $\mathbf\Sigma^0_\alpha(X)$, $\mathbf\Pi^0_\alpha(X)$ and $\mathbf\Delta^0_\alpha(X)$ are multiplicative. 

\begin{lemma} The family $\mathbf \Delta^0_1(X)$ is Polishable if and only if the Polish space $X$ is zero-dimensional.
\end{lemma}

\begin{proof} If the space $X$ is zero-dimensional,  then its topology has a base $\mathcal B$ consisting of clopen (=closed-and-open) sets. For every set $A\in\mathbf \Delta^0_1(X)$ consider the identity map $f:A\to A$ and observe that $\{f[B]:B\in\mathcal B(A)\}\subseteq\mathbf\Delta^0_1(X)$ for the base $\mathcal B(A)\defeq\{B\in\mathcal B:B\subseteq A\}$ of the Polish space $A$, witnessing that the family $\mathbf\Delta^0_1(X)$ is Polishable.

Now assume that the family $\mathbf \Delta^0_1(X)$ of clopen subsets of $X$ is Polishable. Then there exist a Polish space $P$ and a continuous surjective map $f:P\to X$ such that $\{f[B]:B\in\mathcal B\}\subseteq\mathbf\Delta^1_0(X)$ for some base $\mathcal B$ of the topology of $P$. The continuity and surjectivity of $f$ implies that the family $\{f[B]:B\in\mathcal B\}$ is a base of the topology of $X$ that consists of clopen sets and  witnesses that the space $X$ is zero-dimensional.
\end{proof}

Next, using Lemmas~\ref{l:Sigma-Polishable}--\ref{l:Pi}, we shall show that for $\alpha\ge 2$, the classes are $\mathbf\Sigma^0_\alpha(X)$, $\mathbf\Pi^0_\alpha(X)$, $\mathbf\Delta^0_\alpha(X)$ are Polishable. We start with the classes $\mathbf\Sigma^0_\alpha(X)$ and $\mathbf\Delta^0_\alpha(X)$. 

\begin{lemma}\label{l:Sigma-Polishable} Let $\alpha\ge 2$ be a countable ordinal. For every set $A\in\mathbf\Sigma^0_\alpha(X)$, there exists a closed subspace $P\subseteq\w^\w$ and a continuous bijective function $f:P\to A$ such that for any $s\in\w^{<\w}$ the image $f[P\cap\w^\w_s]$ belongs to the class $\mathbf\Delta^0_\alpha(X)$.
\end{lemma}

\begin{proof} Fix a complete metric $d$ generating the topology of the Polish space $X$.

By  \cite[Theorem 22.21]{Ke}, for the set $A\in\mathbf\Sigma^0_\alpha(X)$, there exists an indexed family of sets $(A_s)_{s\in\w^{<\w}}$ such that 
\begin{itemize}
\item[(i)] $A_s\in\mathbf\Delta^0_\alpha(X)$ for every $s\in \w^{<\w}$;
\item[(ii)] $A_\emptyset=A$;
\item[(iii)] for every $s\in\w^{<\w}$, the family $(A_{s\hat{\;}i})_{i\in\w}$ is a disjoint cover of $A_s$ by sets of $d$-diameter $<2^{-|s|}$;
\item[(iv)]  for every $s\in \w^\w$ the intersection $\bigcap_{n\in\w}A_{s{\restriction}n}$ is a singleton if and only if $A_{s{\restriction}n}\ne\emptyset$ for all $n\in\w$.
\end{itemize}

Consider the closed subset $$P=\{s\in \w^\w:\forall n\in\w\;\;(A_{s{\restriction}n}\ne\emptyset)\}$$of $\w^\w$. Let $f:P\to A$ be the function assigning to each sequence $s\in P$ the unique point of the intersection $\bigcap_{n\in\w}A_{s{\restriction}n}$. The conditions (ii)--(iv) guarantee that the map $f:P\to A$ is well-defined, continuous, and bijective. Observe that the countable family $\mathcal B=\{P\cap\w^\w_s:s\in\w^{<\w}\}$ is a base of the topology of the Polish space $P$, and for every $s\in\w^{<\w}$ the image $f[P\cap\w^\w_s]=A_s$ belongs to the class $\mathbf\Delta^0_\alpha(X)$.
\end{proof}

Lemma~\ref{l:Sigma-Polishable} implies that for every countable ordinal $\alpha\ge 2$, the class $\mathbf\Delta^0_\alpha(X)$ is Polishable. Moreover, its Polishability is witnessed by bijective functions. The following lemma proves that the same is true for the additive Borel classes $\mathbf\Sigma^0_\alpha(X)$ with $\alpha\ge 1$.

\begin{lemma}\label{l:Sigma-Poli} For every countable ordinal $\alpha\ge 2$ and set $A\in\mathbf\Sigma^0_\alpha(X)$, there exist a Polish space $P$ and a continuous bijective map $f:P\to A$ such that $\{f[B]:B\in\mathcal B\}\subseteq\mathbf\Sigma^0_\alpha(X)$ for some base $\mathcal B$ of the topology of $P$.
\end{lemma}

\begin{proof} If $\alpha\ge2$, then the existence of $P$, $f$ and $\mathcal B$ follows from Lemma~\ref{l:Sigma-Polishable}. For $\alpha=1$, we can take $P=A$, $f:P\to A$ be the identity function and $\mathcal B$ be any base of the topology of $A$.
\end{proof}

The following lemma establishes the Polishability of the families $\mathbf \Pi^0_\alpha(X)$ for $\alpha\ge 2$.

\begin{lemma}\label{l:Pi} For every countable ordinal $\alpha\ge 2$ and set $A\in\mathbf\Pi^0_\alpha(X)$, there exist a Polish space $P$ and a continuous bijective map $f:P\to A$ such that $\{f[B]:B\in\mathcal B\}\subseteq\mathbf\Pi^0_\alpha(X)$ for some base $\mathcal B$ of the topology of $P$.
\end{lemma}

\begin{proof} Given any set $A\in\mathbf\Pi^0_\alpha(X)$, find a decreasing  sequence of sets $\{A_n\}_{n\in\w}\subseteq\bigcup_{\xi<\alpha}\mathbf\Sigma^0_\xi(X)$ such that $A=\bigcap_{n\in\w}A_n$. By Lemma~\ref{l:Sigma-Poli}, for every $n\in\w$ there exists a Polish space $P_n$ and a continuous bijective function $f_n:P_n\to A_n$ such that $\{f_n[B]:B\in\mathcal B_n\}\subseteq\bigcup_{\xi<\alpha}\mathbf\Sigma^0_\alpha(X)\subseteq\mathbf\Pi^0_\alpha(X)$ for some base $\mathcal B_n\ni P_n$ of the topology of $P_n$.

Consider the closed subspace
$$P=\{(x_n)_{n\in\w}\in\prod_{n\in\w}P_n:\forall n,m\in\w\quad f_n(x_n)=f_m(x_m)\}$$of the Polish space $\prod_{n\in\w}P_n$ and the continuous bijective map 
$$f:P\to A,\quad f:(x_n)_{n\in\w}\mapsto f_0(x_0).$$ Endow the Polish space $P$ with the base
$$\mathcal B=\Big\{P\cap\prod_{n\in\w}B_n:(B_n)_{n\in\w}\in\prod_{n\in\w}\mathcal B_n,\quad|\{n\in\w:B_n\ne P_n\}|<\w\Big\}.$$ For every $B\in\mathcal B$, there exists a sequence $(B_n)_{n\in\w}\in\prod_{n\in\w}\mathcal B_n$ such that the set $F=\{n\in\w:B_n\ne P_n\}$ is finite and $B=P\cap\prod_{n\in\w}B_n$. We claim that $f[B]=A\cap \bigcap_{n\in F}f_n[B_n]\in\mathbf\Pi^0_\alpha(X)$. 

Indeed, for any $x=(x_n)_{n\in\w}\in B$ and $n\in\w$ we have $f(x)=f_n(x_n)\in f_n[P_n]=A_n$ and hence $f(x)\in\bigcap_{n\in\w}A_n=A$. Also for any $n\in F$ we have $x_n\in B_n$ and hence $f(x)=f_n(x_n)\in f_n[B_n]$. Therefore, $f(x)\in A\cap\bigcap_{n\in F}f_n[B_n]$ and $f[B]\subseteq A\cap\bigcap_{n\in F}f_n[B_n]$.

On the other hand, take any $y \in A\cap\bigcap_{n\in F}f_n[B_n]$ and for every $n\in \w$ choose a point $x_n\in f_n^{-1}(y)$ such that $x_n\in B_n$ if $n\in F$. It follows that the sequence $x=(x_n)_{n\in\w}$ belongs to the basic set $B$ and hence $y=f(x)\in f[B]$. This completes the proof of the equality 
 $f[B]=A\cap\bigcap_{n\in F}[B_n]$.
 
 It follows from $\{A\}\cup\{f_n[B_n]:n\in F\}\subseteq\mathbf \Pi^0_\alpha(X)$  that $f[B]=A\cap\bigcap_{n\in F}f_n[B_n]\in\mathbf \Pi^0_\alpha(X)$.
\end{proof}

Lemma~\ref{l:Pi} implies that for every countable ordinal $\alpha\ge 2$ the Borel class $\mathbf\Pi^0_\alpha(X)$ is Polishable. It remains to prove the Polishability of the class $\mathbf\Pi^0_1(X)$ of closed subsets of $X$. Given any set $A\in\mathbf \Pi^0_1(X)$, we conclude that the space $A$ is Polish and hence is the image of a zero-dimensional Polish space $P$ under a closed continuous map $f:P\to X$ (see \cite[4.5.9]{Eng}). Since the space $P$ is zero-dimensional, the family $\mathcal B$ of all clopen sets in $P$ is a base of the topology of $P$. Since the map $f$ is closed, for every $B\in\mathcal B$ the image $f[B]$ is a closed subset of $A$ and hence $f[B]\in\mathbf\Pi^0_1(A)\subseteq\mathbf\Pi^0_1(X)$.
\end{proof}

\begin{remark} By definition, every Polishable family of subsets of a functionally Hausdorff space $X$ consists of analytic subsets of $X$. So, $\mathbf \Sigma^1_1(X)$ is the largest Polishable family on $X$.
\end{remark}

It would be interesting to find examples of non-metrizable spaces $X$ whose algebra of Borel subsets is $\{\emptyset\}$-winning. One of possible candidates are fragmentable compact spaces.
 
Let us recall \cite[5.0.1]{Fab} that a compact Hausdorff space $X$ is {\em fragmentable} if there exists a metric $\rho$ on $X$ such that for every $\e>0$, every nonempty subset $A\subseteq X$ contains a nonempty relatively open subset of $\rho$-diameter $<\e$.

\begin{problem} Let $X$ be a fragmentable compact Hausdorff space and $\I=\{\emptyset\}$. Is the family of Borel subsets of $X$ $\I$-winning?
\end{problem}

\section{Main Technical Result} 

In this section we prove the main technical result of this paper, Theorem~\ref{t:main}. It involves a general version of the Set-Cover game $\Game_{\mathcal S,\C}$, which is played by two players $\mathsf{S}$ and $\mathsf{C}$ on a set $X$, endowed with two families of subsets $\mathcal S$ and  $\mathcal C$ such that $X\in\mathcal S\cap\C$. 

The player $\mathsf S$ starts the game choosing a  set $S_0\in\mathcal S$ and the player $\mathsf C$ answers suggesting a countable cover $\mathcal C_0\subseteq\mathcal C$ of $S_0$. At the $n$-th inning player $\mathsf S$ selects a  set $S_n\in\mathcal S$ with $S_n\prec \mathcal C_{n-1}$ and player $\mathsf C$ answers with a  countable cover $\mathcal C_n\subseteq\mathcal C$ of the set $S_n$. At the end of the game, the player $\mathsf C$ is declared the winner if the intersection $\bigcap_{n\in\w}S_n$ is not empty. Otherwise the player $\mathsf S$ wins the game. This game is called the {\em Set-Cover game} and is denoted by $\Game_{\mathcal S,\C}$.

A {\em strategy} of the player $\mathsf C$ in the game $\Game_{\mathcal S,\mathcal C}$ is a function $\strategy\colon{\mathcal S}^{<\w}\to[\mathcal C]^{\le\w}$ such that  $\{X\}=\strategy(\emptyset)$ and $S_n=\bigcup\strategy(S_0,\dots,S_n)$ for every nonempty finite sequence $(S_0,\dots,S_n)\in\mathcal S^{<\w}$ such that $S_i\prec\strategy(S_0,\dots,S_{i-1})$ for all $i\le n$. 

Let $\strategy\colon{\mathcal S}^{<\w}\to[\mathcal C]^{\le\w}$ be a strategy of the player $\mathsf C$ in the game  $\Game_{\mathcal S,\mathcal C}$. A sequence $(S_n)_{n\in\w}\in\mathcal S^\w$ is called $\strategy$-{\em admissible} if $S_n\prec \strategy(S_0,\dots,S_{n-1})$ for every $n\in\IN$.

A strategy $\strategy\colon\mathcal S^{<\w}\to[\mathcal C]^{\le\w}$ is {\em winning} if for any $\strategy$-admissible sequence $(S_n)_{n\in\w}\in\mathcal S^\w$, the intersection $\bigcap_{n\in\w} S_n$ is not empty.

\begin{theorem}\label{t:main} Let $\I$ be a $\sigma$-ideal on a set $X$ and $\mathcal S,\mathcal C$ be two families of subsets of a set $X\in\s\cap\C$ such that the player $\mathsf C$ has a winning strategy in the game $\Game_{\mathcal S\setminus\I,\mathcal C}$. For every point-finite subfamily $\mathcal J\subseteq\I$ of cardinality $0<|\mathcal J|\le\mathfrak c$, there exists a Cantor scheme $(\mathcal J_s)_{s\in 2^{<\w}}$ with $\mathcal J_\emptyset=\mathcal J$ that has the following properties.
\begin{enumerate}
\item For any set $S\subseteq X$ with $S\cap\bigcup\J\notin\I$, the set $\{t\in 2^{<\w}: S\cap \bigcup\J_t\notin\I\}$ is a perfect subtree of the tree $2^{<\w}$.
\item For any $\sigma\in 2^{<\w}$ and $S\in\mathcal S$ with $S\cap \bigcup\mathcal J_\sigma\notin\I$, there exist a sequence $s\in 2^{<\w}_\sigma$ and a set $C\in\C(S)$ with $C\cap \bigcup\J_s\notin\I$ such that the set $\{t\in 2^{<\w}_s:\mathcal S(C\cap \bigcup\J_t)\not\subseteq\I\}$ is a chain in the tree $2^{<\w}$.
\item For any $\sigma\in 2^{<\w}$ and $S\in\mathcal S$ with $S\cap \bigcup\mathcal J_\sigma\notin\I$, there exist a sequence $s\in 2^{<\w}$ and a set $C\in\C(S)$ such that $\sigma\subset s$, $C\cap\bigcup\J_s\notin\I$ and $\mathcal S(C\cap\bigcup\J_s)\subseteq\I$.
\item For any $\sigma\in 2^{<\w}$ and $S\in\mathcal S(\bigcup\mathcal J_\sigma)\setminus\I$ there exist a sequence $s\in 2^{<\w}_\sigma$ and a set $C\in\C(S\cap \bigcup\J_s)\setminus\I$ such that the set $\{t\in 2^{<\w}_s:\mathcal S(C\cap \bigcup\J_t)\not\subseteq\I\}$ is a chain in the tree $2^{<\w}$.
\item For any $\sigma\in 2^{<\w}$ and $S\in\mathcal S(\bigcup\mathcal J_\sigma)\setminus\I$ there exist a sequence $s\in 2^{<\w}_\sigma$ and a set $C\in\C(S\cap\bigcup\J_s)\setminus\I$ such that either $\mathcal S(C\setminus J)\subseteq\I$ for some $J\in\J$ or $\mathcal S(C\cap \bigcup\J_t)\subseteq\I$ for any sequence $t\in 2^{<\w}$ with $s\subset t$.
\item If the family $\mathcal S$ is multiplicative and $\I$-Lindel\"of, then there exists a decreasing sequence $(\Sigma_n)_{n\in\w}\in(\sigma\mathcal S)^\w$ such that 
\begin{enumerate}
\item for every $n\in\w$ and $s\in 2^n$ we have $\mathcal S(\bigcup\J_s\setminus\Sigma_n)\subseteq\I$;
\item for every $S\in\mathcal S\setminus\I$ there exist $C\in\C(S)\setminus\I$ such that either $\mathcal S(C\cap \Sigma_n)\subseteq\I$ for some $n\in\w$ or $\mathcal S(C\setminus J)\subseteq\I$ for some $J\in\J$.
\end{enumerate}   
\end{enumerate}
\end{theorem}

\begin{proof}  It is well-known \cite[Theorem 11.4]{JW} that the Cantor cube $2^\w$ contains a (Bernstein) subset $\IB$ of cardinality continuum such that $\IB$ contains no uncountable compact subsets. Since $|\mathcal J|\le\mathfrak c$, there exists an injective function $\xi:\J\to \IB$.

The topology of the Cantor cube $2^\w$ is generated by the base $(2^\w_s)_{s\in 2^{<\w}}$ consisting of the clopen sets
$$2^\w_s\defeq\{t\in 2^\w:s\subset t\}\mbox{ \ where \ } s\in 2^{<\w}.$$

For every $s\in 2^{<\w}$, let $\J_s=\xi^{-1}[2^\w_s\cap\IB]$. Endow the set $\J$ with the zero-dimensional metrizable topology generated by the base $\{\J_s:s\in 2^{<\w}\}$. Then the map $\xi:\J\to\IB$ become a topological embedding.

 It is clear that $(\J_s)_{s\in 2^{<\w}}$ is a Cantor scheme with $\J=\J_\emptyset$. In the following six lemmas we shall prove that this Cantor scheme has the properties (1)--(6).
 
By our assumption, the player $\mathsf C$ has a winning strategy $\strategy:(\mathcal S\setminus\I)^{<\w}\to[\C]^{\le\w}$ in the Set-Cover game $\Game_{\mathcal S\setminus\I,\C}$.

\begin{lemma}\label{l:1} For any set $S\subseteq X$ with $S\cap\bigcup\J\notin \I$ the set $T=\{t\in 2^{<\w}: S\cap \bigcup\J_t\notin\I\}$ is a perfect subtree of the tree $2^{<\w}$.
\end{lemma}

\begin{proof} First observe that $T$ contains the empty sequence and hence $T\ne\emptyset$. Given any $s\in T$ we should find two incomparable elements in the set $T\cap 2^{<\w}_s$. Consider the set $$U=\bigcup_{t\in 2^{<\w}_s\setminus T}(S\cap\textstyle\bigcup\J_t)\in\I.$$ It follows from $s\in T$ that $S\cap\bigcup\J_s\notin\I$ and hence $S\cap\bigcup\J_s\setminus U\notin\I$. Choose any point $x_0\in S\cap\bigcup\J_s\setminus U$ and find a set $J_0\in\J_s$ with $x_0\in J_0$. Since $J_0\in\J\subseteq\I$, the set $S\cap\bigcup\J_s\setminus(U\cup J_0)$ does not belong to the $\sigma$-ideal $\I$ and hence contains some point $x_1$. For this point find a set $J_1\in\J_s$ such that $x_1\in J_1$. Since $x_1\in J_1\setminus J_0$, the sets $J_0,J_1$ are distinct and hence there exist sequences $t_0,t_1\in 2^{<\w}_s$ such that $J_0\in\J_{t_0}$, $J_1\in\J_{t_1}$ and $\J_{t_0}\cap \J_{t_1}=\emptyset$.  
It follows that $t_0,t_1$ are incomparable elements of the set $T\cap 2^{<\w}_s$.
\end{proof}  

\begin{lemma}\label{l:2} For any $\sigma\in 2^{<\w}$ and $S\in\mathcal S$ with $S\cap \bigcup\mathcal J_\sigma\notin\I$ there exist a sequence $s\in 2^{<\w}_\sigma$ and a set $C\in\C(S)$ such that $C\cap \bigcup\J_s\notin\I$ and the set $\{t\in 2^{<\w}_s:\mathcal S(C\cap \bigcup\J_t)\not\subseteq\I\}$ is a chain in the tree $2^{<\w}$.
\end{lemma}

\begin{proof} To derive a contradiction, assume that there exists a sequence $\sigma\in 2^{<\w}$ and a set $S\in\mathcal S$ such that $S\cap\bigcup\mathcal J_\sigma\notin\I$ and for any $s\in 2^{<\w}_\sigma$ and set $C\in\C(S)$ with $C\cap \bigcup\J_s\notin \I$, the set $\{t\in 2^{<\w}_s:\mathcal S(C\cap\bigcup\J_t)\not\subseteq\I\}$ is not a chain in $2^{<\w}$. 

Since $S\cap\bigcup\J_\sigma\notin\I$, the countable cover $\strategy(S)$ of $S$ contains a set $C\in\C(S)$ such that $C\cap\bigcup\J_\sigma\notin\I$. By our assumption, the set $\{t\in 2^{<\w}_\sigma:\mathcal S(C\cap\bigcup\J_t)\not\subseteq\I\}$ is not a chain and hence it contains some element $t_\emptyset$. For this element the family $\mathcal S(C\cap\bigcup\J_{t_\emptyset})$ contains some set $S_\emptyset\notin\I$.

\begin{claim}\label{cl:ind-2} There exist indexed families $(S_s)_{s\in 2^{<\w}}$ and $(t_s)_{s\in 2^{<\w}}$ such that $S_\emptyset\subseteq S$, $\sigma\subseteq t_\emptyset$ and for every $n\in \IN$ and $s\in 2^n$ the following conditions are satisfied:
\begin{enumerate}
\item[(i)] $S_s\in \mathcal S(S_{s{\restriction}(n-1)}\cap\bigcup\J_{t_s})\setminus\I$;
\item[(ii)] $S_s\prec\strategy(S_{s{\restriction}0},\dots,S_{s{\restriction}(n-1)})$;
\item[(iii)] $t_{s\hat{\;}0}$ and $t_{s\hat{\;}1}$  are two incomparable elements of the poset $2^{<\w}_{t_s}$.
\end{enumerate}
\end{claim}

\begin{proof} Assume that for some $n\in\w$ and all $s\in \bigcup_{k\le n}2^k$ we have constructed the set $S_s$ and the sequence $t_s$ satisfying the inductive conditions. Given any sequence $s\in 2^n$, we are going to construct sets $S_{s\hat{\;}0}, S_{s\hat{\;}1}$ and sequences $t_{s\hat{\;}0}, t_{s\hat{\;}1}$.

By the inductive condition (i), the set $S_s\in\mathcal S$ has $S_s=S_s\cap\bigcup\J_{t_s}\notin\I$ and then  the countable cover $\strategy(S_{s{\restriction}0},\dots,S_s)\subseteq\C$ of set $S_s$ contains a set $C_s\in \C(S_s)$ such that $C_s\cap\bigcup\J_{t_s}\notin\I$.
By our assumption, the set $\{t\in 2^{<\w}_{t_s}:\mathcal S(C_s\cap\J_t)\not\subseteq\I\}$ contains two incomparable sequences $t_{s\hat{\;}0},t_{s\hat{\;}1}$. For every $i\in\{0,1\}$ the family $\mathcal S(C_s\cap\bigcup\J_{t_{s\hat{\;}i}})\setminus\I$ is not empty and hence contains some set $S_{s\hat{\;}i}$. This completes the inductive step.
 \end{proof}

\begin{claim}\label{cl2:l2} For every sequence $s\in 2^\w$ the intersection $\bigcap_{n\in\w}\J_{t_{s{\restriction}n}}$ contains a unique element $J_s$.
\end{claim}

\begin{proof} Consider the sequence $(S_{s{\restriction}n})_{n\in\w}$. The inductive conditions (i) and (ii) of Claim~\ref{cl:ind-2} ensure that this sequence is $\strategy$-admissible. Since the strategy $\strategy$ is winning, the intersection $\bigcap_{n\in\w}S_{s{\restriction}n}$ is not empty and hence contains some point $a_s$.

For every $n\in\w$, the condition (i) in Claim~\ref{cl:ind-2} ensures that $a_{s}\in S_{s{\restriction}n}\subseteq\bigcup\J_{t_{s{\restriction}n}}$ and hence $a_s\in J_n$ for some set $J_n\in \J_{t_{s{\restriction}n}}$. Since the family $\J$ is point-finite, the family $\{J_n:n\in\w\}\subseteq\{J\in\J:a_s\in J\}$ is finite. Then for some set $J_s\in\J$ the set $\Omega=\{n\in\w:J_n=J_s\}$ is infinite. It follows that for every $n\in\Omega$ we have $J_s=J_n\in \J_{t_{s{\restriction}n}}$. Since the sequence $(t_{s{\restriction}n})_{n\in\w}$ is an infinite chain in the tree $2^{<\w}$, the sequence $(\J_{t_{s{\restriction}n}})_{n\in\w}$ is decreasing and its intersection $\bigcap\J_{t_{s{\restriction}n}}$ coincides with the singleton $\{J_s\}$. 
\end{proof}

The condition (iii) of Claim~\ref{cl:ind-2} and the definition of the topology on the space $\J$ imply that the map 
$$J_*:2^\w\to \J,\quad J_*:s\mapsto J_{\sigma},$$
is continuous and injective. By the compactness of $2^\w$, the map $J_*$ is a topological embedding. Then $\J$ contains the uncountable compact set $J_*[2^\w]$ and the space $\IB$ contains the uncountable compact set $\xi[J_*[2^\w]]$, which is a desirable contradiction completing the proof of Lemma~\ref{l:2}. 
\end{proof}

\begin{lemma}\label{l:3} For any $\sigma\in 2^{<\w}$ and $S\in\mathcal S$ with $S\cap \bigcup\mathcal J_\sigma\notin\I$ there exist a sequence $t\in 2^{<\w}$ and set $C\in \C$ such that $\sigma\subset t$, $C\cap\bigcup\J_t\notin\I$ and $\mathcal S(C\cap\bigcup\J_t)\subseteq\I$.
\end{lemma}

\begin{proof} By Lemma~\ref{l:2}, there exists a sequence $s\in 2^{<\w}_\sigma$ and set $C\in\C(S)$ such that $C\cap\bigcup\J_s\notin\I$ and the set $L=\{t\in 2^{<\w}_s:\mathcal S(C\cap \bigcup\J_t)\not\subseteq\I\}$ is a chain in the tree $2^{<\w}$. By Lemma~\ref{l:1}, the set $T=\{t\in 2^{<\w}:C\cap\bigcup\J_t\notin\I\}$ is a prefect subtree of $2^{<\w}$. Since $s\in T$, there exists a sequence $t\in T\setminus L$ such that $s\subset t$. It follows from $t\in T\setminus L$ that $C\cap\bigcup\J_t\notin\I$ and $\mathcal S(C\cap\bigcup\J_t)\subseteq\I$.
\end{proof}

\begin{lemma}\label{l:4} For any $\sigma\in 2^{<\w}$ and $S\in\mathcal S(\bigcup\mathcal J_\sigma)\setminus\I$ there exist a sequence $s\in 2^{<\w}_\sigma$ and a set $C\in\C(S\cap \bigcup\J_s)\setminus\I$ such that the set $\{t\in 2^{<\w}_s:\mathcal S(C\cap \bigcup\J_t)\not\subseteq\I\}$ is a chain in the tree $2^{<\w}$.
\end{lemma}

\begin{proof} To derive a contradiction, assume that there exists a sequence $\sigma\in 2^{<\w}$ and a set $S\in\mathcal S(\bigcup\mathcal J_\sigma)\setminus\I$ such that for any $s\in 2^{<\w}_\sigma$ and $C\in\C(S\cap\bigcup\J_s)\setminus \I$, the set $\{t\in 2^{<\w}_s:\mathcal S(C\cap\bigcup\J_t)\not\subseteq\I\}$ is not a chain in $2^{<\w}$. 

\begin{claim}\label{cl:l4} There exist indexed families $(S_s)_{s\in 2^{<\w}}$ and $(t_s)_{s\in 2^{<\w}}$ such that $S_\emptyset\subseteq S$, $\sigma\subseteq t_\emptyset$ and for every $n\in \IN$ and $s\in 2^n$ the following conditions are satisfied:
\begin{enumerate}
\item[(i)] $S_s\in \mathcal S(S_{s{\restriction}(n-1)}\cap\bigcup\J_{t_s})\setminus\I$;
\item[(ii)] $S_s\prec\strategy(S_{s{\restriction}0},\dots,S_{s{\restriction}(n-1)})$;
\item[(iii)] $t_{s\hat{\;}0}$ and $t_{s\hat{\;}1}$  are two incomparable elements of the poset $2^{<\w}_{t_s}$.
\end{enumerate}
\end{claim}

\begin{proof} To start the inductive construction, use the choice of the sequence $\sigma$ and for the set $C=S$ find a sequence $t_\emptyset\in 2^{<\w}_\sigma$ such that the family $\mathcal S(C\cap\bigcup\J_{t_\emptyset})\setminus\I$ is not empty and hence contains some set $S_\emptyset$. 

Now assume that for some $n\in\w$ and all $s\in \bigcup_{k\le n}2^k$ we have constructed the set $S_s$ and the sequence $t_s$ satisfying the inductive conditions. Given any sequence $s\in 2^n$, we are going to construct sets $S_{s\hat{\;}0}, S_{s\hat{\;}1}$ and sequences $t_{s\hat{\;}0}, t_{s\hat{\;}1}$.

Since $S_s\notin\I$, the countable cover $\strategy(S_{s{\restriction}0},\dots,S_s)\subseteq\C$ of set $S_s$ contains a set $C_s\in \C(S_s)\setminus\I\subseteq\C(S\cap\bigcup\J_{t_s})\setminus\I$.

By our assumption, the set $\{t\in 2^{<\w}_{t_s}:\mathcal S(C_s\cap\bigcup\J_t)\not\subseteq\I\}$ contains two incomparable sequences $t_{s\hat{\;}0},t_{s\hat{\;}1}$. For every $i\in\{0,1\}$ the family $\mathcal S(C_s\cap\J_{t_{s\hat{\;}i}})\setminus\I$ is not empty and hence contains some set $S_{s\hat{\;}i}$. This completes the inductive step.
 \end{proof}

Repeating the arguments of Claims~\ref{cl:ind-2} and \ref{cl2:l2},  we can show that for every sequence $s\in 2^\w$ the intersection $\bigcap_{n\in\w}\J_{t_{s{\restriction}n}}$ contains a unique element $J_s$, and the function $J_*:2^\w\to\J$, $J^*:s\mapsto J_s$, is injective and continuous. Then $\xi\circ J_*[2^\w]$ is an uncountable compact subspace of the space $\IB$, which contradicts the choice of $\IB$.
\end{proof}

\begin{lemma}\label{l:5} For any $\sigma\in 2^{<\w}$ and $S\in\mathcal S(\bigcup\mathcal J_\sigma)\setminus\I$, there exist a sequence $s\in 2^{<\w}_\sigma$ and a set $C\in\C(S\cap\bigcup\J_s)\setminus\I$ such that either $\mathcal S(C\setminus J)\subseteq\I$ for some $J\in\J$ or $\mathcal S(C\cap \bigcup\J_t)\subseteq\I$ for any sequence $t\in 2^{<\w}$ with $s\subset t$.
\end{lemma}

\begin{proof} To derive a contradiction, assume there exist $\sigma\in 2^{<\w}$ and $S\in\mathcal S(\bigcup\mathcal J_\sigma)\setminus\I$ such that for any  $s\in 2^{<\w}_\sigma$ and $C\in\C(S\cap\bigcup\J_s)\setminus\I$ we have: $\mathcal S(C\setminus J)\not\subseteq\I$ for all $J\in\J$, and $\mathcal S(C\cap \bigcup\J_t)\not\subseteq\I$ for some sequence $t\in 2^{<\w}$ with $s\subset t$.

By Lemma~\ref{l:4}, there exists a sequence $s\in 2^{<\w}_\sigma$ and a set $C\in\mathcal C(S\cap\bigcup\J_s)\setminus\I$ such that the set $L=\{t\in 2^{<\w}_s:\mathcal S(C\cap\bigcup\J_t)\not\subseteq\I\}$ is a chain in the tree $2^{<\w}$.
Our assumption guarantees that the chain $L$ is infinite, which implies that $\bigcap_{t\in L}\J_t$ is either empty or a singleton. In any case we can find a set $J_\infty\in\J$ such that $\bigcap_{t\in L}\J_t\subseteq\{J_\infty\}$. By our assumption the family $\mathcal S(C\setminus J_\infty)\setminus\I$ contains some set $S_0$. Let $t_0=s$.

Inductively we shall construct sequences $(S_n)_{n\in\IN}$ and $(t_n)_{n\in\IN}$ such that for every $n\in\IN$ the following conditions are satisfied:
\begin{enumerate}
\item[(i)] $S_n\subseteq \mathcal S(S\cap S_{n-1}\cap\bigcup\J_{t_n})\setminus\I$;
\item[(ii)] $S_n\prec\strategy(S_0,\dots,S_{n-1})$;
\item[(iii)] $t_{n-1}\subset t_n$.
\end{enumerate}

Assume that for some $n\in\IN$ a sequence $t_{n-1}\supset t_0=s$ and a set $S_{n-1}\in\mathcal S(S\cap \bigcup\J_{t_{n-1}})\setminus\I$ have been constructed. Since $S_{n-1}\notin\I$, the countable cover $\strategy(S_0,\dots,S_{n-1})$ of $S_{n-1}$ contains a set $C_n\in\C(S_{n-1})\setminus\I$. Since $C_n\in\mathcal C(S\cap\bigcup \J_{t_{n-1}})\setminus\I$, our assumption yields a sequence $t_n\in 2^{<\w}$ such that $t_{n-1}\subset t_n$ and the family $\mathcal S(C_n\cap \bigcup\J_{t_n})\setminus\I$ contains some set $S_n$. This completes the inductive step.
\smallskip

Since the strategy $\strategy$ is winning, the intersection $\bigcap_{n\in\w}S_n$ contains some point $x$. For every $n\in\w$ the inductive condition (i) implies $x\in S_n\subseteq\bigcup\J_{t_n}$ and hence $x\in J_n$ for some $J_n\in\J_{t_n}$. Since the family $\J$ is point-finite, the family $\{J_n:n\in\w\}\subseteq\{J\in\J:x\in J\}$ is finite. Consequently, there exists $J_\infty'\in\J$ such that the set $\Omega=\{n\in\w:J_n=J_\infty'\}$ is infinite. Observe that for every $n\in\Omega$ we have $J'_\infty=J_n\in\J_{t_n}$, which implies that $\{J'_\infty\}=\bigcap_{n\in\Omega}\J_{t_n}=\bigcap_{n\in\w}\J_{t_n}=\{J_\infty\}$ and hence $J'_\infty=J_\infty$. Then $x\in J'_\infty=J_\infty$. On the other hand, $x\in S_0\subseteq X\setminus J_\infty$ by the choice of the set $S_0$. This contradiction completes the proof of Lemma~\ref{l:5}.
\end{proof}

\begin{lemma}\label{l:6}If the family $\mathcal S$ is multiplicative and $\I$-Lindel\"of, then there exists a decreasing sequence $(\Sigma_n)_{n\in\w}\in(\sigma\mathcal S)^\w$ such that 
\begin{enumerate}
\item[(a)] for every $n\in\w$ and $s\in 2^n$ we have $\mathcal S(\bigcup\J_s\setminus\Sigma_n)\subseteq\I$;
\item[(b)] for every $S\in\mathcal S\setminus\I$ there exist $C\in\C(S)\setminus\I$ such that either $\mathcal S(C\cap \Sigma_n)\subseteq\I$ for some $n\in\w$ or $\mathcal S(C\setminus J)\subseteq\I$ for some $J\in\J$.
\end{enumerate}   
\end{lemma}

\begin{proof} By the $\I$-Lindel\"of property of the family $\mathcal S$, for every $s\in 2^{<\w}$ there exists a set $\Sigma_s\subseteq \sigma\mathcal S$ such that $\Sigma_s\subseteq \bigcup\J_s$ and $\mathcal S(\bigcup\J_s\setminus \Sigma_s)\subseteq\I$.

For every $n\in\w$ let 
$$\Sigma_{n}=\bigcup_{m=n}^\infty\bigcup_{s\in 2^m}\Sigma_s\in\sigma\mathcal S.$$ It is clear that the sequence $(\Sigma_n)_{n\in\w}$ is decreasing. 

Observe that for every $n\in\w$ and $s\in 2^n$ we have $\Sigma_s\subseteq \Sigma_n$ and hence 
$$\textstyle{\mathcal S(\bigcup\J_s\setminus \Sigma_n)\subseteq \mathcal S(\bigcup\J_s\setminus \Sigma_s)\subseteq\I,}$$by the choice of the set $\Sigma_s$. This shows that the condition (a) holds. It remains to verify the condition (b).

 Assume that this condition does not hold, which means that there exists $S\in \mathcal S\setminus\I$ such that for any $C\in\mathcal C(S)\setminus\I$, $n\in\w$ and $J\in\J$ we have $\mathcal S(C\cap \Sigma_n)\not\subseteq\I$ and $\mathcal S(C\setminus J)\not\subseteq \I$. The last condition implies that $\mathcal S(C)\not\subseteq\I$ for any set $C\in\C(S)\setminus\I$.
 
Let $S_0=S$ and $s_0=\emptyset$. We inductively construct sequences $(S_n)_{n\in \IN}$, $(s_n)_{n\in \IN}$  such that for every $n\in\IN$ the following conditions are satisfied:
\begin{enumerate}
\item[(i)] $|s_n|>|s_{n-1}|$;
\item[(ii)] $S_n\in \mathcal S(S\cap S_{n-1}\cap\Sigma_{s_n})\setminus\I$;
\item[(iii)] $S_n\prec \strategy(S_0,\dots,S_{n-1})$;
\item[(iv)] $\mathcal S(S_n\cap\bigcup \J_t)\subseteq\I$ for any $t\in 2^{<\w}$ with $s_n\subset t$.
\end{enumerate}
 
Assume that for some $n\in\IN$ we have constructed sets $S_0,\dots,S_{n-1}$ and sequences\break $s_0,\dots,s_{n-1}$ satisfying the inductive conditions (i)--(iv).  Since $S_{n-1}\notin\I$, the countable cover $\strategy(S_0,\dots,S_{n-1})\subseteq\C$ of $S_{n-1}$ contains a set $C_n\in\C(S_{n-1})\setminus\I$. Let $m_n=1+|s_{n-1}|$. By the choice of the set $S$, the family $\mathcal S(C_n\cap\Sigma_{m_n})\setminus\I$ contains some set $S_n'$. Since $S_n'=S'_n\cap \Sigma_{m_n}=\bigcup\{S_n'\cap \Sigma_t:t\in 2^{<\w},\;|t|\ge m_n\}$, there exists $t_n\in 2^{<\w}$ such that $|t_n|\ge m_n$ and $S_n'\cap \Sigma_{t_n}\notin\I$. Since the family $\mathcal S$ is multiplicative and $\Sigma_{t_n}\in\sigma\mathcal S$, there exists a set $S_n''\in\mathcal S(S_n'\cap\Sigma_{t_n})\setminus\I$. Since $S_n''\in \mathcal S(\bigcup\J_{t_n})\setminus\I$, we can apply Lemma~\ref{l:5} and find a sequence $s_n\in 2^{<\w}_{t_n}$ and a set $C'_n\in\C(S_n''\cap\bigcup\J_{s_n})\setminus\I$ such that either $\mathcal S(C'_n\setminus J)\subseteq\I$ for some $J\in\J$ or $\mathcal S(C'_n\cap\bigcup\J_{t})\subseteq\I$ for any $t\in 2^{<\w}$ with $s_n\subset t$. By the choice of the set $S$, the first alternative is impossible. Consequently,   $\mathcal S(C'_n\cap\bigcup\J_{t})\subseteq\I$ for any $t\in 2^{<\w}$ with $s_n\subset t$. By the choice of the set $S$, the family $\mathcal S(C_n')\setminus\I$ contains some set $S_n$. It is easy to see that $S_n$ and $s_n$ satisfy the inductive conditions (i)--(iv).
\smallskip

Since the strategy $\strategy$ is winning, the intersection $\bigcap_{n\in\w}S_n$ contains some point $x$. For every $n\in\w$ we have $x\in S_n\subseteq\bigcup\J_{s_n}$ and hence $x\in J_n$ for some $J_n\in\J_{s_n}$. Since the family $\J$ is point-finite, the set $\{J_n:n\in\w\}\subseteq\{J\in\J:x\in J\}$ is finite. Consequently,  there exists $J\in\J$ such that the set $\Omega=\{n\in\w:J=J_n\}$ is infinite. Choose any numbers $n<m$ in $\Omega$ and observe that $J=J_n=J_m\in \J_{s_n}\cap\J_{s_m}$ and $|s_n|<|s_m|$ imply $s_n\subset s_m$. Then $S_m\in\mathcal S(S_n\cap \Sigma_{s_m})\setminus\I\subseteq\mathcal S(S_n\cap \bigcup\J_{s_m})\setminus\I$, which contradicts the inductive condition (iv).
\end{proof}
\end{proof}

\section{Some implications of the main Theorem~\ref{t:main}}\label{s:=>}

Theorem~\ref{t:main} and the definition of an $\I$-winning family imply the following theorem.

\begin{theorem}\label{t:main-W} Let $\I$ be a $\sigma$-ideal on a set $X$ and $\A$ be an $\I$-winning family of subsets of $X$. For every point-finite subfamily $\mathcal J\subseteq\I$ of cardinality $0<|\mathcal J|\le\mathfrak c$, there exists a Cantor scheme $(\mathcal J_s)_{s\in 2^{<\w}}$ with $\mathcal J_\emptyset=\mathcal J$ that has the following properties.
\begin{enumerate}
\item For any set $S\subseteq X$ with $S\cap\bigcup\J\notin\I$, the set $\{t\in 2^{<\w}: S\cap \bigcup\J_t\notin\I\}$ is a perfect subtree of the tree $2^{<\w}$.
\item For any $\sigma\in 2^{<\w}$ and $S\in\A$ with $S\cap \bigcup\mathcal J_\sigma\notin\I$, there exist a sequence $s\in 2^{<\w}_\sigma$ and a set $C\in\A(S)$ with $C\cap \bigcup\J_s\notin\I$ such that the set $\{t\in 2^{<\w}_s:\A(C\cap \bigcup\J_t)\not\subseteq\I\}$ is a chain in the tree $2^{<\w}$.
\item For any $\sigma\in 2^{<\w}$ and $S\in\A$ with $S\cap \bigcup\mathcal J_\sigma\notin\I$, there exist a sequence $s\in 2^{<\w}$ and a set $C\in\A(S)$ such that $\sigma\subset s$, $C\cap\bigcup\J_s\notin\I$ and $\A(C\cap\bigcup\J_s)\subseteq\I$.
\item For any $\sigma\in 2^{<\w}$ and $S\in\A(\bigcup\mathcal J_\sigma)\setminus\I$ there exist a sequence $s\in 2^{<\w}_\sigma$ and a set $C\in\A(S\cap \bigcup\J_s)\setminus\I$ such that the set $\{t\in 2^{<\w}_s:\A(C\cap \bigcup\J_t)\not\subseteq\I\}$ is a chain in the tree $2^{<\w}$.
\item For any $\sigma\in 2^{<\w}$ and $S\in\A(\bigcup\mathcal J_\sigma)\setminus\I$ there exist a sequence $s\in 2^{<\w}_\sigma$ and a set $C\in\A(S\cap\bigcup\J_s)\setminus\I$ such that either $\A(C\setminus J)\subseteq\I$ for some $J\in\J$ or $\A(C\cap \bigcup\J_t)\subseteq\I$ for any sequence $t\in 2^{<\w}$ with $s\subset t$.
\item If the family $\A$ is multiplicative and $\I$-Lindel\"of, then there exists a decreasing sequence $(\Sigma_n)_{n\in\w}\in(\sigma\A)^\w$ such that 
\begin{enumerate}
\item for every $n\in\w$ and $s\in 2^n$ we have $\A(\bigcup\J_s\setminus\Sigma_n)\subseteq\I$;
\item for every $S\in\A\setminus\I$ there exist $C\in\A(S)\setminus\I$ such that either $\A(C\cap \Sigma_n)\subseteq\I$ for some $n\in\w$ or $\A(C\setminus J)\subseteq\I$ for some $J\in\J$.
\end{enumerate}   
\end{enumerate}
\end{theorem}

We shall apply Theorem~\ref{t:main-W}(3,5) to prove the following extension of Theorem~\ref{t:main1}.

\begin{corollary}\label{c:7.3} Let $\I$ be a $\sigma$-ideal on a set $X$ of cardinality $|X|\le\mathfrak c$ and $\mathcal A$ be an $\I$-winning family of subsets of $X$ such that $X\in\A\setminus\I$. Any point-finite family $\J\subseteq\I$ with  $\bigcup\J\notin\I$ (and $\bigcup\J=X$) contains a subfamily $\J'\subseteq\J$ such that for some set $C\in\A$ we have $C\cap\bigcup\J'\notin\I$ but $\A(C\cap \bigcup\J')\subseteq\I$ (and $\A(C\setminus\bigcup\J')\subseteq\I$).
\end{corollary}

\begin{proof} Let $\J\subseteq\I$ be a point-finite family with $\bigcup\J\notin\I$. Since $\bigcup\J\ne\emptyset$ and $|\bigcup\J|\le|X|\le\mathfrak c$, the point-finite family $\J$ has cardinality $0<|\J|\le\mathfrak c$. 

By Theorem~\ref{t:main-W}, there exists a Cantor scheme $(\J_s)_{s\in 2^{<\w}}$ with $\J_\emptyset=\J$ satisfying the conditions (1)--(6) of Theorem~\ref{t:main-W}. 

Applying Theorem~\ref{t:main-W}(3) to the empty sequence $\sigma=\emptyset\in 2^0$ and set $S=X\in\A$, we can find a sequence $s\in 2^{<\w}$ and a set $C\in\A$ such that $C\cap\bigcup\J_s\notin\I$ and $\A(C\cap\bigcup\J_s)\subseteq\I$. Then the subfamily $\J'\defeq\J_s$ has the required property.

If $X=\bigcup\J$, then applying Theorem~\ref{t:main-W}(5) to the empty sequence $\sigma=\emptyset\in 2^0$ and set $S=X\in\A$, we can find a sequence $s\in 2^{<\w}$ and a set $C\in\A(\J_s)\setminus\I$ such that either $\A(C\setminus J)\subseteq\I$ for some $J\in\J$ or $\A(C\cap\bigcup\J_t)\subseteq\I$ for any sequence $t\in 2^{<\w}$ with $s\subset t$. In the first case we put $\J'\defeq\{J\}$ and obtain that $\A(C\setminus\bigcup\J')=\A(C\setminus J)\subseteq\I$ and $\A(C\cap\bigcup\J')=\A(C\cap J)\subseteq\I$. In the second case, put $\J'\defeq\J_{s\hat{\;}0}$ and observe that $\A(C\cap\bigcup\J')\subseteq\I$ and $\A(C\setminus\bigcup\J_{s\hat{\;}0})\subseteq\A(C\cap\bigcup\J_{s\hat{\;}1})\subseteq\I$.
\end{proof}

Next, we apply Theorem~\ref{t:main-W} to prove two corollaries implying Theorems~\ref{t:measure} and \ref{t:category}, announced in the introduction.

\begin{corollary}\label{c:measure} Let $\mu$ be a non-atomic $\sigma$-additive probability Borel measure on a Borel  space $X$ and $\I\defeq\{A\subseteq X:\mu^*(A)=0\}$. For any point-finite subfamily $\J\subseteq\I$, there exists a Cantor scheme $(\mathcal J_s)_{s\in 2^{<\w}}$ with $\J_\emptyset=\J$ and a decreasing sequence of Borel sets $(A_n)_{n\in\w}$ in $X$ such that
\begin{enumerate}
\item $\lim_{n\to\infty}\mu(A_n)=0$;
\item for any $n\in\w$ and $s\in 2^n$ the set $\bigcup\J_s\setminus A_n$  contains no $\mu$-positive Borel subsets of $X$.
\end{enumerate}
\end{corollary}

\begin{proof} By Theorem~\ref{t:Polishable} and Propositions~\ref{p:Borel2} and \ref{p:Borel}, the family $\mathbf\Delta^1_1(X)$ of all Borel subsets of $X$ is $\I$-winning. 
Since $\I=\{A\subseteq X:\mu^*(A)=0\}$, the family $\mathbf\Delta^1_1(X)$ is $\I$-ccc and hence $\I$-Lindel\"of. 
By Theorem~\ref{t:main-W}(6), there exists a Cantor scheme $(\J_s)_{s\in 2^{<\w}}$ with $\J_\emptyset=\J$ and a decreasing sequence of Borel sets $(A_n)_{n\in\w}$ such that 
\begin{itemize}
\item[(a)] for any $n\in\w$ and $s\in 2^n$ the set $\bigcup\J_s\setminus A_n$ contains no $\mu$-positive Borel subsets of $X$;
\item[(b)] for any $\mu$-positive Borel set $S\subseteq X$ there exists $n\in\w$ such that $\mu(S\setminus A_n)>0$.
\end{itemize}
The conditions (a) and (b) imply the statements (2) and (1), respectively. 
\end{proof}

\begin{corollary}\label{c:category} Let $\I$ be the $\sigma$-ideal of meager sets in a Polish space $X$. For any point-finite subfamily $\J\subseteq\I$, there exists a Cantor scheme $(\mathcal J_s)_{s\in 2^{<\w}}$ with $\J_\emptyset=\J$ and a decreasing sequence of regular closed sets  $(B_n)_{n\in\w}$ in $X$ such that
\begin{enumerate}
\item for any non-meager Borel set $A\subseteq X$ there exists $n\in\w$ such that $A\setminus B_n$ is non-meager;
\item $\bigcap_{n\in\w}B_n$ is nowhere dense in $X$;
\item for any $n\in\w$ and $s\in 2^n$ the set $\bigcup\J_s\setminus B_n$  contains no non-meager Borel subset of $X$;
\item If $X$ is compact, then for every neighborhood $U\subseteq X$ of $\bigcap_{n\in\w}B_n$ there exists $n\in\w$ such that $B_n\subseteq U$.
\end{enumerate}
\end{corollary} 

\begin{proof}  By Theorem~\ref{t:Polishable} and Propositions~\ref{p:Borel2} and \ref{p:Borel}, the family $\mathbf\Delta^1_1(X)$ of all Borel subsets of $X$ is $\I$-winning. Since each Borel subset of the Polish space $X$ has the Baire property, the family $\mathbf \Delta^1_1(X)$ is $\I$-ccc and hence $\I$-Lindel\"of. By Theorem~\ref{t:main-W}(6),
there exists a decreasing sequence of Borel sets $(A_n)_{n\in\w}$ such that 
\begin{enumerate}
\item[(i)] for any non-meager Borel set $A\subseteq X$ there exists $n\in\w$ such that $A\setminus A_n$ is non-meager;
\item[(ii)] for any $n\in\w$ and $s\in 2^n$ the set $\bigcup\J_s\setminus A_n$ contains no non-meager Borel subsets of $X$.
\end{enumerate}

For every $n\in\w$ let $B_n$ be the set of all points $x\in X$ such that for every neighborhood $O_x\subseteq X$ of $x$ the intersection $O_x\cap A_n$ is not meager in $X$. It is easy to see that the set $B_n$ is regular closed and $A_n\triangle B_n$ is meager. The monotonicity of the sequence $(A_n)_{n\in\w}$ implies that the sequence $(B_n)_{n\in\w}$ is decreasing. The properties (i) and (ii) imply the properties (1) and (3), respectively.
The property (2) follows from (1) and the closedness of the sets $B_n$. 

 If $X$ is compact, then the condition (4) follows from Theorem 3.10.2(3) in \cite{Eng}.
\end{proof} 

\section{Marczewski--Burstin representations}\label{s:mb}

In this section we apply Theorem~\ref{t:main} to families and ideals that admit a Marczewski--Burstin representation. Such representations were studied e.g. in \cite{BBC}, \cite{BBK}. 

\begin{definition} Given any family $\F$ of nonempty subsets of a set $X$, consider the families
$$	\mathcal{S}_0(\F)\defeq\{A\subseteq X: \forall F\in \mathcal{F}\quad\F(F\setminus A)\ne\emptyset\}$$
and
$$\mathcal{S}(\mathcal{F})\defeq\{A\subseteq X: \forall F\in \mathcal{F}\quad\F(F\cap A)\cup\F(F\setminus A)\ne\emptyset\}.$$
\end{definition}
Note that $\mathcal S(\mathcal F)$ is an algebra of sets and $\mathcal S_0(\F)$ is an ideal on $X$.

\begin{definition} Let $\mathcal B$ be an algebra of subsets of a set $X$ and $\I\subseteq\mathcal B$ be an ideal. The pair $(\mathcal B,\I)$ is defined to have a {\em Marczewski--Burstin representation} (briefly, an {\em MB-representation}) if there a family $\F$ of nonempty subsets of $X$ such that $$\mathcal{S}_0(\F)=\mathcal{I} \text{ and } \mathcal{S}(\mathcal{F})=\mathcal{B}.$$
\end{definition}

\begin{example} 
	\begin{enumerate}
		\item For the family $\F$ of closed sets of nonzero Lebesgue measure on the real line, the ideal $\s_0(\F)$ coincides with the $\sigma$-ideal $\mathcal N$ of Lebesgue null sets and the algebra $\s(\F)$ coincides with the $\sigma$-algebra of Lebesgue measurable sets in $\IR$.
		\item For the family of nonmeager $G_\delta$-subsets of $\IR$, the ideal $\s_0(\F)$ coincides with the $\sigma$-ideal $\M$  of meager subsets of $\IR$ and $\s(\mathcal F)$ coincides with the $\sigma$-algebra of subsets with the Baire property in $\IR$.
		\item For the family $\F$ of uncountable compact subsets of $\IR$ the ideal $\s_0(\F)$ coincides with the Marczewski $\sigma$-ideal $s_0$ and the algebra $\s(\F)$ with the $\sigma$-algebra $s$ of Marczewski measurable sets, see \cite{Marczewski}.
	\end{enumerate}
\end{example}

\begin{remark} Some examples algebras and ideals without MB-representations can be found in \cite{BBC} and \cite{BBK}.
\end{remark}


\begin{theorem}\label{MB}
	Let $\F$ be a family of nonempty Borel subset in a Polish space $X$ such that $\s_0(\F)$ is a $\sigma$-ideal and the algebra $\s(\F)$ contains all Borel subsets of $X$. Any point-finite family $\J\subseteq\s_0(\F)$ with $\bigcup\J\notin\s_0(\F)$ contains a subfamily $\J'$ such that $\bigcup\J'\notin\s(\F)$.
\end{theorem}

\begin{proof} Applying Theorem~\ref{t:main-B} to the $\sigma$-ideal $\I=\s_0(\F)$, we can find a subfamily $\J'\subseteq\J$ and a Borel set $B$ in $X$ such that the set $B'\defeq B\cap\bigcup\J'$ does not belong to the ideal $\I$ and contains no $\I$-positive Borel sets.  It follows from  $B'\notin \I=s_0(\F)$ that $\F(F\setminus B')=\emptyset$ for some $F\in\F$. Since $B\in\s(\F)$, there exists $E\in\F$ such that $E\subseteq F\cap B$ or $E\subseteq F\setminus B$. The second case is impossible as $B'\subseteq B$ and $\F(F\setminus B)\subseteq\F(F\setminus B')=\emptyset$. So, $E\subseteq F\cap B$ and $$\textstyle E\setminus\bigcup\J'=E\cap F\cap B\setminus \bigcup\J'\subseteq F\cap B\setminus\bigcup\J'=F\cap B\setminus(B\cap\bigcup\J')\subseteq  F\setminus B',$$ which implies $\F(E\setminus\bigcup\J')\subseteq\F(F\setminus B')=\emptyset$.

Assuming that $\bigcup\J'\in\s(\F)$, we can find a set $H\in\F$ such that $H\subseteq E\cap\bigcap\J'$ or $H\subseteq E\setminus\bigcap\J'$. The second case is impossible as $\F(E\setminus\bigcup\J')=\emptyset$. So, $H\subseteq E\cap\bigcap\J'$. Since the set $B\cap\bigcup\J'\supseteq E\cap\bigcap\J'$ contains no $\I$-positive Borel subsets, the set $H\in\F$ belongs to the ideal $\I=\s_0(\F)$, which is not possible as $\F\cap\s_0(\F)=\emptyset$. This contradiction shows that $\bigcup\J'\in\s(\F)$.
\end{proof}

Every perfect set in a Polish space can be represented as the body of a perfect tree. This observation alows us to modify the classical construction of Marczewski ideal   $s_0$ and a corresponding algebra of $s$-measurable sets and obtain similar pairs for various classes of trees. Let us recall some examples.

\begin{definition}\label{trees}	A tree $T\se \w^{<\w}$ is called
	\begin{itemize}
		\item  {\em perfect} or {\em Sacks} if $(\forall \sigma\in T)(\exists \tau\in T)(\sigma\subseteq \tau\land (\exists n\neq m)( \tau^{\frown} n\in T \land \tau^{\frown} m\in T))$; 
		\item {\em superperfect} or {\em Miller} if $(\forall \sigma\in T)(\exists \tau\in T)(\sigma\subseteq \tau\land (\exists^\infty n)( \tau^{\frown}n\in T))$;
		\item {\em Laver} if  $(\exists \sigma \in T)(\forall \tau\in T)(\tau\se\sigma \lor (\sigma\se\tau \land (\exists^\infty n)( \tau^{\frown}n\in T)))$. 
		\end{itemize}
\end{definition}

For any tree $T\se \w^{<\w}$ define a body of $T$ as
$$
[T] = \{ x\in \w^\w:\; \forall n\in \w \;\; x{\restriction}n \in T \}.
$$
Now, define pairs
\begin{itemize}
	\item $(m,m_0) = (\s(\F),\s_0(\F))$ where $\F$ is a family of all bodies of Miller trees, 
	\item $(l,l_0) = (\s(\F),\s_0(\F))$ where $\F$ is a family of all bodies of Laver trees.
\end{itemize}

The ideals $s_0,\ m_0,\ l_0$ were studied e.g. in \cite{Marczewski}, \cite{GSS}, \cite{JMSS}, \cite{MRZ}.  The fusion argument ensures that each family $s_0,\ m_0,\ l_0$ forms a $\sigma$-ideal. Moreover, every analytic set is $s$-measurable, $m$-measurable and $l$-measurable (see \cite{Mi}).  Applying Theorem~\ref{MB} to the pairs $(s,s_0)$, $(m,m_0)$, $(l,l_0)$ we obtain the following corollary.

\begin{corollary} Let $(t,t_0)\in \{(s,s_0), (m,m_0), (l,l_0)\}.$
Any point-finite family $\mathcal{J} \subseteq t_0$ with $\bigcup\mathcal{J} \notin t_0$ contains  a subfamily $\mathcal{J}'\subseteq\mathcal{J}$ such that $\bigcup\mathcal{J}'\notin t$.
\end{corollary}

Now, let us consider the space $[\w]^\w$ endowed with the {\em Ellentuck topology}, which is generated by the base consisting of the sets  
$$
[a,A]=\{ B\in[\w]^\w:\; a\se B\se a\cup (A\setminus[0,\max a])\}
$$
where $a\in[\w]^{<\w}$ and $A\in[\w]^\w$.

A subset $X\se [\w]^\w$ is defined to be {\em completely Ramsey} if for any $a\in[\w]^{<\w}$ and $A\in[\w]^\w$ there exists $B\in[\w]^\w$ such that $B\se A$ and $[a,B]\se X$ or $[a,B]\cap X =\emptyset.$ 

A subset $X\se [\w]^\w$ is called {\em Ramsey null} if  for any $a\in[\w]^{<\w}$ and  $A\in[\w]^\w$ there exists $B\in[\w]^\w$ such that $B\se A$ and $[a,B]\cap X =\emptyset.$ 

By $RN$ and $CR$ we denote the families of all Ramsey null and completely Ramsey subsets of $[\w]^\w$, respectively. 

The pair $(CR,RN)$ is MB-represented by the family
$$
\F = \{ [a,A]:\; a\in [\w]^{<\w} \land A\in [\w]^\w \}.
$$

The following result was obtained earlier in \cite[Corollary 3.9]{Prikry}.

\begin{corollary}[Koumoullis, Prikry] 
Any point-finite family $\J$ of  Ramsey null sets with $\bigcup\J\notin RN$ contains  subfamily $\J'\subseteq\J$ whose union $\bigcup \J'$ is not completely Ramsey. 
\end{corollary}

\begin{proof} By the famous Ellentuck Theorem  (see \cite{E} or \cite{Matet}), the ideal $RN$ coincides with the $\sigma$-ideal of meager sets in the Ellentuck topology on $[\w]^\w$ and the algebra $CR$ coincides with the $\sigma$-algebra of sets with the Baire property in the Ellentuck topology on $[\w]^\w$. 

Besides the Ellentuck topology, the space $[\w]^\w$ carries a natural Polish topology, generated by the base consisting of the sets $$[a,b]\defeq\{A\in[\w]^{<\w}:A\cap b=a\},$$
where $a\subseteq b$ are finite subsets of $\w$. Since this Polish topology is contained in the Ellentuck topology, the algebra $CR$ contains all Borel sets (in the Polish topology). Moreover, every set in the family 
$$
\F = \{ [a,A]:\; a\in [\w]^{<\w} \land A\in [\w]^\w \}
$$
is closed, and hence Borel, in the Polish topology on $[\w]^\w$.

Now we can  apply Theorem~\ref{MB} to the family $\F$ and finish the proof.
\end{proof}

\end{document}